\documentclass{svjour3}                     
\smartqed  
\usepackage{graphicx}
\usepackage{amsmath,amssymb,amsfonts}
\usepackage{algorithm,algorithmic}
\usepackage{url}
\usepackage{color}
\usepackage{multirow}

\newcommand{\R}{\mathbb{R}}
\renewcommand{\Re}{\mathbb{R}}
\newcommand{\gf}{\nabla f}

\newcommand{\<}{\left\langle}
\renewcommand{\>}{\right\rangle}

\newcommand{\comment}[1]{} 

\DeclareMathOperator{\tr}{tr}

\DeclareMathOperator{\sgn}{sgn}

\renewcommand{\dfrac}{\displaystyle \frac} 

\begin{document}

\title{The Augmented Lagrange Multiplier Method for\\
Exact Recovery of Corrupted Low-Rank Matrices\thanks{This is only
a technical report. {\bf Please cite ``Zhouchen Lin, Risheng Liu,
and Zhixun Su, Linearized Alternating Direction Method with
Adaptive Penalty for Low Rank Representation, NIPS 2011."}
(available at http://arxiv.org/abs/1109.0367) instead for a more
general method called Linearized Alternating Direction Method.
Thanks!
}
}


\author{Zhouchen Lin \and
        Minming Chen \and
        Yi Ma
}

\authorrunning{Z. Lin et al.} 

\institute{Zhouchen Lin \at
                Key Laboratory of Machine Perception, School of
                EECS, Peking University.
              \email{zlin@pku.edu.cn}           
           \and
           Minming Chen \at
              Institute of Computing Technology, Chinese Academy of Sciences,
              China.
           \and
           Yi Ma \at
              Visual Computing Group, Microsoft Research Asia, Beijing 100190, China, and Electrical \& Computer Engineering Department, University of Illinois at Urbana-Champaign, USA.
}

\date{Received: date / Accepted: date}

\maketitle

\begin{abstract}
This paper proposes scalable and fast algorithms for solving the Robust PCA problem, namely recovering a low-rank matrix with an unknown fraction of
its entries being arbitrarily corrupted. This problem arises in many
applications, such as image processing, web data ranking, and
bioinformatic data analysis. It was recently shown that under
surprisingly broad conditions, the Robust PCA problem can be
exactly solved via convex optimization that minimizes a
combination of the nuclear norm and the $\ell^1$-norm . In this
paper, we apply the method of augmented Lagrange multipliers (ALM)
to solve this convex program. As the objective function is non-smooth,
we show how to extend the classical analysis of ALM to such new objective
functions and prove the optimality of the proposed algorithms and characterize
their convergence rate. Empirically, the proposed new algorithms can be more
than \emph{five times faster} than the previous state-of-the-art
algorithms for Robust PCA, such as the accelerated proximal
gradient (APG) algorithm. Moreover, the new algorithms achieve
higher precision, yet being less storage/memory demanding. We also show
that the ALM technique can be used to solve the (related but
somewhat simpler) matrix completion problem and obtain rather
promising results too. Matlab code of all algorithms discussed are
available at \url{http://perception.csl.illinois.edu/matrix-rank/home.html}
\end{abstract}

\section{Introduction}
Principal Component Analysis (PCA), as a popular tool for
high-dimensional data processing, analysis, compression, and
visualization, has wide applications in scientific and engineering
fields \cite{Jolliffe1986}. It assumes that the given
high-dimensional data lie near a much lower-dimensional linear
subspace. To large extent, the goal of PCA is to efficiently and
accurately estimate this low-dimensional subspace.

Suppose that the given data are arranged as the columns of a large
matrix $D \in \Re^{m \times n}$. The mathematical model for
estimating the low-dimensional subspace is to find a low rank
matrix $A$, such that the discrepancy between $A$ and $D$ is
minimized, leading to the following constrained optimization:
\begin{eqnarray}
\label{eqn:pca} \min_{A,E} & \|E\|_F, \quad \mbox{subject to} &
 \quad \mathrm{rank}(A) \leq r, \quad D = A + E,
\end{eqnarray}
where $r\ll \min(m,n)$ is the target dimension of the subspace and
$\|\cdot\|_F$ is the Frobenius norm, which corresponds to assuming
that the data are corrupted by i.i.d. Gaussian noise. This problem
can be conveniently solved by first computing the Singular Value
Decomposition (SVD) of $D$ and then projecting the columns of $D$
onto the subspace spanned by the $r$ principal left singular
vectors of $D$ \cite{Jolliffe1986}.

As PCA gives the optimal estimate when the corruption is caused by
additive i.i.d. Gaussian noise, it works well in practice as long
as the magnitude of noise is small. However, it breaks down under
large corruption, even if that corruption affects only very few of
the observations. In fact, even if only one entry of $A$ is
arbitrarily corrupted, the estimated $\hat A$ obtained by
classical PCA can be arbitrarily far from the true $A$. Therefore,
it is necessary to investigate whether a low-rank matrix $A$ can
still be efficiently and accurately recovered from a corrupted
data matrix $D = A + E$, where some entries of the additive errors
$E$ may be {\em arbitrarily large}.

Recently, Wright et al. \cite{Wright2009} have shown that under
rather broad conditions the answer is affirmative: as long as
the error matrix $E$ is sufficiently sparse (relative to the rank
of $A$), one can exactly recover the low-rank matrix $A$ from $D =
A + E$ by solving the following convex optimization problem:
\begin{eqnarray}
\label{eqn:rpca} \min_{A,E} & \|A\|_* + \lambda\,\|E\|_1, \quad
\mbox{subject to} &  \quad D = A + E,
\end{eqnarray}
where $\|\cdot \|_*$ denotes the nuclear norm of a matrix (i.e.,
the sum of its singular values), $\|\cdot \|_1$ denotes the sum of
the absolute values of matrix entries, and $\lambda$ is a positive
weighting parameter. Due to the ability to exactly recover
underlying low-rank structure in the data, even in the presence of
large errors or outliers, this optimization is referred to as {\em
Robust PCA} (RPCA) in \cite{Wright2009} (a popular term that has
been used by a long line of work that aim to render PCA robust to
outliers and gross corruption). Several applications of RPCA, e.g.
background modeling and removing shadows and specularities from
face images, have been demonstrated in \cite{Wright-NIPS2009} to
show the advantage of RPCA.

The optimization (\ref{eqn:rpca}) can be treated as a general
convex optimization problem and solved by any off-the-shelf
interior point solver (e.g., CVX \cite{Boyd-cvx}), after being
reformulated as a semidefinite program \cite{Chandrasekharan2009}.
However, although interior point methods normally take very few
iterations to converge, they have difficulty in handling large
matrices because the complexity of computing the step direction is
$O(m^6)$, where $m$ is the dimension of the matrix. As a result,
on a typical personal computer (PC) generic interior point solvers
cannot handle matrices with dimensions larger than $m = 10^2$. In
contrast, applications in image and video processing often involve
matrices of dimension $m = 10^4$ to $10^5$; and applications in
web search and bioinformatics can easily involve matrices of
dimension $m = 10^6$ and beyond. So the generic interior point
solvers are too limited for Robust PCA to be practical for many
real applications.

That the interior point solvers do not scale well for large
matrices is because they rely on second-order information of the
objective function. To overcome the scalability issue, we should
use the first-order information only and fully harness the special
properties of this class of convex optimization problems. For
example, it has been recently shown that the (first-order)
iterative thresholding (IT) algorithms can be very efficient for
$\ell^1$-norm minimization problems arising in compressed sensing
\cite{Yin2008,Beck2009,Yin2008-SJIS,Cai2009-MC}.  It has also been
shown in \cite{Cai2008} that the same techniques can be used to
minimize the nuclear norm for the matrix completion (MC) problem,
namely recovering a low-rank matrix from an incomplete but clean
subset of its entries \cite{Recht2008-SR,Candes2008}.

As the matrix recovery (Robust PCA) problem \eqref{eqn:rpca}
involves minimizing a combination of both the $\ell^1$-norm and
the nuclear norm, in the original paper \cite{Wright2009},  the
authors have also adopted the iterative thresholding technique to
solve \eqref{eqn:rpca} and obtained similar convergence and
scalability properties. However, the iterative thresholding scheme
proposed in \cite{Wright2009} converges extremely slowly.
Typically, it requires about $10^4$ iterations to converge, with
each iteration having the same cost as one SVD. As a result, even
for matrices with dimensions as small as $m = 800$, the algorithm
has to run 8 hours on a typical PC. To alleviate the slow
convergence of the iterative thresholding method
\cite{Wright2009}, Lin et al. \cite{Lin2009} have proposed two new
algorithms for solving the problem \eqref{eqn:rpca}, which in some
sense complementary to each other: The first one is an accelerated
proximal gradient (APG) algorithm applied to the primal, which is
a direct application of the FISTA framework introduced by
\cite{Beck2009}, coupled with a fast continuation
technique\footnote{Similar techniques have been applied to the
matrix completion problem by \cite{Toh2009}.}; The second one is a
gradient-ascent algorithm applied to the dual of the problem
\eqref{eqn:rpca}. From simulations with matrices of dimension up
to $m = 1,000$, both methods are at least 50 times faster than the
iterative thresholding method (see \cite{Lin2009} for more
details).

In this paper, we present novel algorithms for matrix recovery
which utilize techniques  of augmented Lagrange multipliers (ALM).
The exact ALM (EALM) method to be proposed here is proven to have
a pleasing Q-linear convergence speed, while the APG is in theory
only sub-linear. A slight improvement over the exact ALM leads an
inexact ALM (IALM) method, which converges practically as fast as
the exact ALM, but the required number of partial SVDs is
significantly less. Experimental results show that IALM is at
least \emph{five times faster} than APG, and its precision is also
higher. In particular, the number of non-zeros in $E$ computed by
IALM is much more accurate (actually, often exact) than that by
APG, which often leave many small non-zero terms in $E$.

In the rest of the paper, for completeness, we will first sketch
the previous work in Section \ref{sec:Previous}. Then we present
our new ALM based algorithms and analyze their convergence
properties in Section \ref{sec:ALM} (while leaving all technical proofs
to Appendix A). We will also quickly
illustrate how the same ALM method can be easily adapted to solve
the (related but somewhat simpler) matrix completion (MC) problem.
We will then discuss some implementation details of our algorithms
in Section \ref{sec:Discuss}. Next in Section \ref{sec:Experiments},
we compare the new algorithms and other existing algorithms for both matrix recovery  and matrix completion, using extensive simulations on randomly generated matrices. Finally we give some concluding remarks in Section
\ref{sec:Conclusions}.

\section{Previous Algorithms for Matrix Recovery}
In this section, for completeness as well as purpose of
comparison, we briefly introduce and summarize other existing
algorithms for solving the matrix recovery problem
\eqref{eqn:rpca}.

\label{sec:Previous}
\subsection{The Iterative Thresholding Approach}
The IT approach proposed in \cite{Wright2009} solves a relaxed convex problem
of (\ref{eqn:rpca}):
\begin{equation}
\label{eqn:IT} \min_{A,E} \quad \|A\|_* + \lambda\,\|E\|_1 +
\frac{1}{2\tau}\,\|A\|_F^2 + \frac{1}{2\tau}\,\|E\|_F^2, \quad
\mbox{subject to}\quad A+E=D,
\end{equation}
where $\tau$ is a large positive scalar so that the objective
function is only perturbed slightly. By introducing a Lagrange
multiplier $Y$ to remove the equality constraint, one has the
Lagrangian function of (\ref{eqn:IT}):
\begin{equation}
\label{eqn:IT_Lagrange} L(A,E,Y) = \|A\|_* + \lambda\,\|E\|_1 +
\frac{1}{2\tau}\,\|A\|_F^2 + \frac{1}{2\tau}\,\|E\|_F^2+
\frac{1}{\tau}\langle Y, D-A-E\rangle.
\end{equation}
Then the IT approach updates $A$, $E$ and $Y$ iteratively. It
updates $A$ and $E$ by minimizing $L(A,E,Y)$ with respect to $A$
and $E$, with $Y$ fixed. Then the amount of violation of the
constraint $A+E=D$ is used to update $Y$.

For convenience, we introduce the following soft-thresholding
(shrinkage) operator:
\begin{equation}
\mathcal{S}_\varepsilon[x] \doteq \left\{ \begin{array}{ll} x -
\varepsilon, & \mbox{if }x > \varepsilon, \\ x + \varepsilon, &
\mbox{if }x < - \varepsilon,
\\ 0, & \text{otherwise},
\end{array} \right.
\end{equation}
where $x \in \Re$ and $\varepsilon
> 0$. This operator can be extended to vectors and matrices by applying it
element-wise. Then the IT approach works as described in Algorithm
\ref{alg:IT}, where the thresholdings directly follow from the
well-known analysis \cite{Cai2008,Yin2008}:
\begin{equation}
U\mathcal{S}_\varepsilon[S]V^T = \arg\min\limits_X
\varepsilon\|X\|_*+\dfrac{1}{2}\|X-W\|_F^2,\quad
\mathcal{S}_\varepsilon[W] =\arg\min\limits_X \varepsilon
\|X\|_1+\dfrac{1}{2}\|X-W\|_F^2,
\end{equation}
where $USV^T$ is the SVD of $W$. Although being extremely
simple and provably correct, the IT algorithm
requires a very large number of iterations to converge and it is
difficult to choose the step size $\delta_k$ for speedup, hence
its applicability is limited.
\begin{algorithm}[h]
\caption{\bf (RPCA via Iterative Thresholding)}
\begin{algorithmic}[1]
\REQUIRE Observation matrix $D \in \R^{m \times n}$, weights
$\lambda$ and $\tau$. \WHILE{not converged} \STATE $(U,S,V) =
\mbox{svd}(Y_{k-1})$, \STATE $A_k = U \mathcal{S}_{\tau}[S] V^T$,
\STATE $E_k = \mathcal{S}_{\lambda \tau}[Y_{k-1}]$, \STATE $Y_k =
Y_{k-1}+\delta_k(D-A_k-E_k)$. \ENDWHILE \ENSURE $A \leftarrow
A_k$, $E \leftarrow E_k$.
\end{algorithmic}
\label{alg:IT}
\end{algorithm}

\subsection{The Accelerated Proximal Gradient Approach}
\label{sec:apg} A general theory of the accelerated proximal
gradient approach can be found in
\cite{Tseng2008-SJO,Beck2009,Nesterov1983-SMD}. To solve the
following unconstrained convex problem:
\begin{equation}
\label{eqn:PG} \min_{X\in\mathcal{H}} \quad F(X) \; \doteq \; g(X)
+ f(X),
\end{equation}
where $\mathcal{H}$ is a real Hilbert space endowed with an inner
product $\langle \cdot, \cdot \rangle$ and a corresponding norm
$\|\cdot\|$, both $g$ and $f$ are convex and $f$ is further
Lipschitz continuous: $\| \gf(X_1) - \gf(X_2) \| \le L_f \| X_1 -
X_2 \|$, one may approximate $f(X)$ locally as a quadratic
function and solve
\begin{equation}
\label{eqn:approx}
X_{k+1}=\arg\min\limits_{X\in\mathcal{H}}Q(X,Y_k) \doteq f(Y_k) +
\langle \gf(Y_k), X-Y_k \rangle + \frac{L_f}{2} \|X-Y_k\|^2 +
g(X),
\end{equation}
which is assumed to be easy, to update the solution $X$. The
convergence behavior of this iteration depends strongly on the
points $Y_k$ at which the approximations $Q(X,Y_k)$ are formed.
The natural choice $Y_k = X_k$ (proposed, e.g., by
\cite{Fukushima1981-IJSS}) can be interpreted as a gradient
algorithm, and results in a convergence rate no worse than
$O(k^{-1})$ \cite{Beck2009}. However, for smooth $g$ Nesterov
showed that instead setting $Y_k = X_k + \tfrac{t_{k-1} - 1}{t_k}
(X_k - X_{k-1})$ for a sequence $\{t_k\}$ satisfying
$t_{k+1}^2-t_{k+1} \leq t_k^2$ can improve the convergence rate to
$O(k^{-2})$ \cite{Nesterov1983-SMD}. Recently, Beck and Teboulle
extended this scheme to the nonsmooth $g$, again demonstrating a
convergence rate of $O(k^{-2})$, in a sense that $F(X_k) - F(X^*)
\leq Ck^{-2}$ \cite{Beck2009}.

The above accelerated proximal gradient approach can be directly
applied to a relaxed version of the RPCA problem, by identifying
$$X=(A,E),\quad f(X)=\frac{1}{\mu}\|D-A-E\|_F^2,\quad\mbox{and}\quad g(X)=\|A\|_*+\lambda \|E\|_1,$$
where $\mu$ is a small positive scalar. A continuation technique
\cite{Toh2009}, which varies $\mu$, starting from a large initial
value $\mu_0$ and decreasing it geometrically with each iteration
until it reaches the floor $\bar \mu$, can greatly speed up the
convergence. The APG approach for RPCA is described in Algorithm
\ref{alg:apg} (for details see \cite{Lin2009,Wright-NIPS2009}).

\begin{algorithm}[h]
\caption{\bf (RPCA via Accelerated Proximal Gradient)}
\begin{algorithmic}[1]
\REQUIRE Observation matrix $D \in \R^{m \times n}$, $\lambda$.
\STATE $A_0= A_{-1}= 0$; $E_0= E_{-1} = 0$; $t_0=t_{-1} = 1$;
$\bar{\mu} > 0$; $\eta < 1$. \WHILE{not converged} \STATE $Y_k^A =
A_k + \frac{t_{k-1}-1}{t_k}\,(A_k-A_{k-1})$, $Y_k^E = E_k +
\frac{t_{k-1}-1}{t_k}\,(E_k-E_{k-1})$. \STATE $G_k^A = Y_k^A -
\frac{1}{2}\,\left( Y_k^A + Y_k^E - D\right)$. \STATE $(U,S,V) =
\mathrm{svd}(G_k^A)$, $A_{k+1} = U
\mathcal{S}_{\frac{\mu_k}{2}}[S] V^T$. \STATE $G_k^E = Y_k^E -
\frac{1}{2}\,\left( Y_k^A + Y_k^E - D\right)$. \STATE $E_{k+1} =
\mathcal{S}_{\frac{\lambda \mu_k}{2}}[G_k^E]$. \STATE $t_{k+1} =
\frac{1+\sqrt{4t_k^2+1}}{2}$; $\mu_{k+1} =
\max(\eta\,\mu_k,\bar{\mu})$. \STATE $k\gets k+1$. \ENDWHILE
\ENSURE $A \gets A_k$, $E \gets E_k$.
\end{algorithmic}
\label{alg:apg}
\end{algorithm}

\subsection{The Dual Approach}
\label{sec:dual} The dual approach proposed in our earlier work \cite{Lin2009} tackles the problem (\ref{eqn:rpca}) via its dual. That is, one first solves the dual problem
\begin{equation}
\label{eqn:Dual}
\begin{array}{c}
\displaystyle\max_{Y} \;  \langle D, Y\rangle, \quad \mbox{subject
to} \quad J(Y) \leq 1,
\end{array}
\end{equation}
for the optimal Lagrange multiplier $Y$, where
\begin{equation}
 \langle A, B\rangle = \tr(A^TB),\quad J(Y) = \max \left( \|Y\|_2 ,\lambda^{-1} \|Y\|_\infty
 \right),
\end{equation}
and $\|\cdot\|_\infty$ is the maximum absolute value of the matrix
entries. A steepest ascend algorithm constrained on the surface
$\{Y|J(Y)=1\}$ can be adopted to solve (\ref{eqn:Dual}), where the
constrained steepest ascend direction is obtained by projecting
$D$ onto the tangent cone of the convex body $\{Y|J(Y)\leq 1\}$.
It turns out that the optimal solution to the primal problem
(\ref{eqn:rpca}) can be obtained during the process of finding the
constrained steepest ascend direction. For details of the final algorithm,
one may refer to \cite{Lin2009}.

A merit of the dual
approach is that only the principal singular space associated to
the largest singular value 1 is needed. In theory, computing this
special principal singular space should be easier than computing
the principal singular space associated to the \emph{unknown}
leading singular values. So the dual approach is promising if an
efficient method for computing the principal singular space
associated to the \emph{known largest} singular value can be
obtained.

\section{The Methods of Augmented Lagrange Multipliers}\label{sec:ALM}
In \cite{Bertsekas-LM}, the general method of augmented
Lagrange multipliers is introduced for solving constrained optimization
problems of the kind:
\begin{equation}
\min f(X),\quad\mbox{subject to}\quad h(X)=0,
\end{equation}
where $f:\mathbb{R}^n\rightarrow \mathbb{R}$ and $h:\mathbb{R}^n
\rightarrow \mathbb{R}^m$. One may define the augmented Lagrangian
function:
\begin{equation}
L(X,Y,\mu) = f(X) + \langle Y, h(X)\rangle +
\dfrac{\mu}{2}\|h(X)\|_F^2,
\end{equation}
where $\mu$ is a positive scalar, and then the optimization problem can
be solved via the method of augmented Lagrange multipliers, outlined as Algorithm
\ref{alg:general_ALM} (see \cite{Bertsekas-NP} for more details).

\begin{algorithm}[th]
\caption{\bf (General Method of Augmented Lagrange Multiplier)}
\begin{algorithmic}[1]
\STATE $\rho \geq 1$. \WHILE{not converged} \STATE Solve $X_{k+1}
= \arg\min\limits_{X} L(X,Y_k,\mu_k)$. \STATE $Y_{k+1} = Y_k +
\mu_k h(X_{k+1})$; \STATE Update $\mu_k$ to $\mu_{k+1}$. \ENDWHILE
\ENSURE $X_k$.
\end{algorithmic}
\label{alg:general_ALM}
\end{algorithm}

Under some rather general conditions, when $\{\mu_k\}$ is an
increasing sequence and both $f$ and $h$ are \emph{continuously
differentiable} functions, it has been proven in
\cite{Bertsekas-LM} that the Lagrange multipliers $Y_k$ produced
by Algorithm \ref{alg:general_ALM} converge Q-linearly to the
optimal solution when $\{\mu_k\}$ is bounded and super-Q-linearly
when $\{\mu_k\}$ is unbounded. This superior convergence property
of ALM makes it very attractive. Another merit of ALM is that the
optimal step size to update $Y_k$ is proven to be the chosen
penalty parameter $\mu_k$, making the parameter tuning much easier
than the iterative thresholding algorithm. A third merit of ALM is
that  the algorithm converges to the exact optimal solution, even
without requiring $\mu_k$ to approach infinity
\cite{Bertsekas-LM}. In contrast, strictly speaking both the
iterative thresholding and APG approaches mentioned earlier only
find approximate solutions for the problem. Finally, the analysis
(of convergence) and the implementation of the ALM algorithms are
relatively simple, as we will demonstrate on both the matrix
recovery and matrix completion problems.

\subsection{Two ALM Algorithms for Robust PCA (Matrix Recovery)}
For the RPCA problem \eqref{eqn:rpca}, we may apply the augmented
Lagrange multiplier method by identifying:
$$X=(A,E),\quad f(X)=\|A\|_*+\lambda \|E\|_1,\quad\mbox{and}\quad h(X)=D-A-E.$$ Then the Lagrangian function is:
\begin{equation}
L(A,E,Y,\mu) \doteq \|A\|_* + \lambda \|E\|_1 + \langle Y, D-A-E
\rangle + \dfrac{\mu}{2}\|D-A-E\|_F^2,
\end{equation}
and the ALM method for solving the RPCA problem can be described
in Algorithm \ref{alg:exact_ALM}, which we will refer to as the
exact ALM (EALM) method, for reasons that will soon become clear.

The initialization $Y_0^*=\sgn(D)/J(\sgn(D))$ in the algorithm is
inspired by the dual problem (\ref{eqn:Dual}) as it is likely to
make the objective function value $\langle D, Y_0^*\rangle$
reasonably large.

\begin{algorithm}[th]
\caption{\bf (RPCA via the Exact ALM Method)}
\begin{algorithmic}[1]
\REQUIRE Observation matrix $D \in \R^{m \times n}$, $\lambda$.
\STATE $Y_0^*=\sgn(D)/J(\sgn(D))$; $\mu_0 > 0$; $\rho > 1$; $k =
0$. \WHILE{not converged} \STATE \COMMENT {Lines 4-12 solve
$(A_{k+1}^*, E_{k+1}^*) = \arg\min\limits_{A,E}
L(A,E,Y_k^*,\mu_k)$.} \STATE $A_{k+1}^{0} = A_{k}^{*}$,
$E_{k+1}^{0} = E_{k}^{*}$, $j = 0$; \WHILE{not converged} \STATE
\COMMENT {Lines 7-8 solve $A_{k+1}^{j+1} = \arg\min\limits_{A}
L(A,E_{k+1}^j,Y^*_k,\mu_k)$.} \STATE $(U,S,V) =
\mbox{svd}(D-E_{k+1}^{j}+\mu_k^{-1} Y_k^*)$; \STATE $A_{k+1}^{j+1}
= U\mathcal{S}_{\mu_k^{-1}}[S]V^T$; \STATE \COMMENT {Line 10
solves $E_{k+1}^{j+1} = \arg\min\limits_{E}
L(A_{k+1}^{j+1},E,Y^*_k,\mu_k)$.} \STATE $E_{k+1}^{j+1} =
\mathcal{S}_{\lambda\mu_k^{-1}}[D-A_{k+1}^{j+1}+\mu_k^{-1}
Y_k^*]$; \STATE $j\gets j+1$. \ENDWHILE \STATE $Y_{k+1}^* = Y_k^*
+ \mu_k (D - A_{k+1}^* - E_{k+1}^*)$. \STATE Update $\mu_k$ to
$\mu_{k+1}$. \STATE $k \gets k+1$. \ENDWHILE \ENSURE
$(A_k^*,E_k^*)$.
\end{algorithmic}
\label{alg:exact_ALM}
\end{algorithm}

Although the objective function of the RPCA problem
\eqref{eqn:rpca} is non-smooth and hence the results in
\cite{Bertsekas-LM} do not directly apply here, we can still prove
that Algorithm \ref{alg:exact_ALM} has the same excellent
convergence property. More precisely, we have established
the following statement.
\begin{theorem}\label{thm:exact_ALM}
For Algorithm \ref{alg:exact_ALM}, any accumulation point
$(A^*,E^*)$ of $(A_k^*,E_k^*)$ is an optimal solution to the RPCA
problem and the convergence rate is at least $O(\mu_k^{-1})$ in
the sense that $$\left|\|A_k^*\|_* + \lambda \|E_k^*\|_1 -f^*
\right|=O(\mu_{k-1}^{-1}),$$ where $f^*$ is the optimal value of
the RPCA problem.
\end{theorem}
\begin{proof}
See Appendix \ref{sec:exact_ALM}. \qquad \end{proof}
From Theorem
\ref{thm:exact_ALM}, we see that if $\mu_k$ grows geometrically,
the EALM method will converge Q-linearly; and if $\mu_k$ grows
faster, the EALM method will also converge faster. However,
numerical tests show that for larger $\mu_k$, the iterative
thresholding approach to solve the sub-problem $(A_{k+1}^*,
E_{k+1}^*) = \arg\min\limits_{A,E} L(A,E,Y_k^*,\mu_k)$ will
converge slower. As the SVD accounts for the majority of the
computational load, the choice of $\{\mu_k\}$ should be judicious
so that the total number of SVDs is minimal.

Fortunately, as it turns out, we do not have to solve the
sub-problem $$(A_{k+1}^*, E_{k+1}^*) = \arg\min\limits_{A,E}
L(A,E,Y_k^*,\mu_k)$$ exactly. Rather, updating $A_{k}$ and $E_{k}$
once when solving this sub-problem is sufficient for $A_{k}$ and
$E_{k}$ to converge to the optimal solution of the RPCA problem.
This leads to an inexact ALM (IALM) method, described in
Algorithm \ref{alg:inexact_ALM}.
\begin{algorithm}[th]
\caption{\bf (RPCA via the Inexact ALM Method)}
\begin{algorithmic}[1]
\REQUIRE Observation matrix $D \in \R^{m \times n}$, $\lambda$.
\STATE $Y_0=D/J(D)$; $E_0=0$; $\mu_0 > 0$; $\rho > 1$; $k = 0$.
\WHILE{not converged} \STATE \COMMENT {Lines 4-5 solve $A_{k+1} =
\arg\min\limits_{A} L(A,E_{k},Y_k,\mu_k)$.} \STATE $(U,S,V) =
\mbox{svd}(D-E_{k}+\mu_k^{-1} Y_k)$; \STATE $A_{k+1} =
U\mathcal{S}_{\mu_k^{-1}}[S]V^T$. \STATE \COMMENT {Line 7 solves
$E_{k+1} = \arg\min\limits_{E} L(A_{k+1},E,Y_k,\mu_k)$.} \STATE
$E_{k+1} = \mathcal{S}_{\lambda\mu_k^{-1}}[D-A_{k+1}+\mu_k^{-1}
Y_k]$. \STATE $Y_{k+1} = Y_k + \mu_k (D - A_{k+1} - E_{k+1})$.
\STATE Update $\mu_k$ to $\mu_{k+1}$. \STATE $k \gets k+1$.
\ENDWHILE \ENSURE $(A_k,E_k)$.
\end{algorithmic}
\label{alg:inexact_ALM}
\end{algorithm}

The validity and optimality of Algorithm \ref{alg:inexact_ALM} is guaranteed by
the following theorem.
\begin{theorem}\label{thm:inexact_ALM}
For Algorithm \ref{alg:inexact_ALM}, if $\{\mu_k\}$ is
nondecreasing and $\sum\limits_{k=1}^{+\infty}\mu_k^{-1}=+\infty$
then $(A_k,E_k)$ converges to an optimal solution $(A^*,E^*)$ to
the RPCA problem.
\end{theorem}
\begin{proof}
See Appendix \ref{sec:inexact_ALM}. \qquad \end{proof} We can
further prove that the condition
$\sum\limits_{k=1}^{+\infty}\mu_k^{-1}=+\infty$ is also necessary
to ensure the convergence, as stated in the following theorem.
\begin{theorem}\label{thm:necessary_condition}
If $\sum\limits_{k=1}^{+\infty}\mu_k^{-1}<+\infty$ then the
sequence $\{(A_k,E_k)\}$ produced by Algorithm
\ref{alg:inexact_ALM} may not converge to the optimal solution of
the RPCA problem.
\end{theorem}
\begin{proof}
See Appendix \ref{sec:necessary_condition}. \qquad \end{proof}
Note that, unlike Theorem \ref{thm:exact_ALM} for the exact ALM
method, Theorem \ref{thm:inexact_ALM} only guarantees convergence
but does not specify the rate of convergence for the inexact ALM
method. The condition
$\sum\limits_{k=1}^{+\infty}\mu_k^{-1}=+\infty$ implies that
$\mu_k$ cannot grow too fast. The choice of $\mu_k$ will be
discussed in detail in Section \ref{sec:Discuss}.

\subsection{An ALM Algorithm for Matrix Completion} The
matrix completion (MC) problem can be viewed as a special case of
the matrix recovery problem, where one has to recover the missing
entries of a matrix, given limited number of known entries. Such a
problem is ubiquitous, e.g., in machine learning
\cite{Abernethy2006,Amit2007,Argyriou2007}, control
\cite{Mesbahi1997} and computer vision \cite{Tomasi1992}. In many
applications, it is reasonable to assume that the matrix to
recover is of low rank. In a recent paper \cite{Candes2008},
Cand{\`{e}s} and Recht proved that most matrices $A$ of rank $r$
can be perfectly recovered by solving the following optimization
problem:
\begin{eqnarray}
\label{eqn:MC} \min_{A} \|A\|_*, \quad \mbox{subject to}
 \quad A_{ij}=D_{ij},\quad \forall (i,j)\in\Omega,
\end{eqnarray}
provided that the number $p$ of samples obeys $p\geq C r
n^{6/5}\ln n$ for some positive constant $C$, where $\Omega$ is
the set of indices of samples. This bound has since been improved by
the work of several others. The state-of-the-art algorithms to
solve the MC problem (\ref{eqn:MC}) include the APG approach
\cite{Toh2009} and the singular value thresholding (SVT) approach
\cite{Cai2008}. As the RPCA problem is closely connected to the MC
problem, it is natural to believe that the ALM method can be similarly
effective on the MC problem.

We may formulate the MC problem as follows
\begin{eqnarray}
\label{eqn:MC_reformulate} \min_{A} \|A\|_*, \quad \mbox{subject
to}
 \quad A+E=D,\quad \pi_\Omega(E)=0,
\end{eqnarray}
where $\pi_\Omega:\mathbb{R}^{m\times n}\rightarrow
\mathbb{R}^{m\times n}$ is a linear operator that keeps the
entries in $\Omega$ unchanged and sets those outside $\Omega$
(i.e., in $\bar{\Omega}$) zeros. As $E$ will compensate for the
unknown entries of $D$, the unknown entries of $D$ are simply set
as zeros. Then the \emph{partial} augmented Lagrangian function
(Section 2.4 of \cite{Bertsekas-LM}) of (\ref{eqn:MC_reformulate})
is
\begin{equation}\label{eqn:MC_ALM}
L(A,E,Y,\mu)=\|A\|_* + \langle Y,D-A-E \rangle +
\frac{\mu}{2}\|D-A-E\|_F^2.
\end{equation}
Then similarly we can have the exact and inexact ALM approaches
for the MC problem, where for updating $E$ the constraint
$\pi_\Omega(E)=0$ should be enforced when minimizing
$L(A,E,Y,\mu)$. The inexact ALM approach is described in Algorithm
\ref{alg:MC-IALM}.

\begin{algorithm}[th]
\caption{\bf (Matrix Completion via the Inexact ALM Method)}
\begin{algorithmic}[1]
\REQUIRE Observation samples $D_{ij}$, $(i,j)\in\Omega$, of matrix
$D \in \R^{m \times n}$. \STATE $Y_0=0$; $E_0=0$; $\mu_0
> 0$; $\rho
> 1$; $k = 0$. \WHILE{not converged} \STATE \COMMENT {Lines 4-5 solve $A_{k+1} =
\arg\min\limits_{A} L(A,E_{k},Y_k,\mu_k)$.} \STATE $(U,S,V) =
\mbox{svd}(D-E_{k}+\mu_k^{-1} Y_k)$; \STATE $A_{k+1} =
U\mathcal{S}_{\mu_k^{-1}}[S]V^T$. \STATE \COMMENT {Line 7 solves
$E_{k+1} = \arg\min\limits_{\pi_\Omega(E)=0}
L(A_{k+1},E,Y_k,\mu_k)$.} \STATE $E_{k+1} =
\pi_{\bar{\Omega}}(D-A_{k+1}+\mu_k^{-1} Y_k)$. \STATE $Y_{k+1} =
Y_k + \mu_k (D - A_{k+1} - E_{k+1})$. \STATE Update $\mu_k$ to
$\mu_{k+1}$. \STATE $k \gets k+1$. \ENDWHILE \ENSURE $(A_k,E_k)$.
\end{algorithmic}
\label{alg:MC-IALM}
\end{algorithm}
Note that due to the choice of $E_k$, $\pi_{\bar{\Omega}}(Y_k) =0$
holds throughout the iteration, i.e., the values of $Y_k$ at
unknown entries are always zeros. Theorems \ref{thm:exact_ALM} and
\ref{thm:inexact_ALM} are also true for the matrix completion
problem. As the proofs are similar, we hence omit them here.
Theorem \ref{thm:necessary_condition} is also true for the matrix
completion problem, because it is easy to verify that
$Y_k=\pi_{\Omega}(\hat{Y}_k)$. As $\{\hat{Y}_k\}$ is bounded (cf.
Lemma \ref{lemma:bound_Y}), $\{Y_k\}$ is also bounded. So the
proof of Theorem \ref{thm:necessary_condition} is still valid for
the matrix completion problem.

\section{Implementation Details}\label{sec:Discuss}
\label{sec:implementation_issues}

\paragraph{Predicting the Dimension of Principal
Singular Space.} \label{sec:svd} It is apparent that computing the
full SVD for the RPCA and MC problems is unnecessary: we only need
those singular values that are larger than a particular threshold
and their corresponding singular vectors. So a software package,
PROPACK \cite{Propack}, has been widely recommended in the
community. To use PROPACK, one have to predict the dimension of
the principal singular space whose singular values are larger than
a given threshold. For Algorithm \ref{alg:inexact_ALM}, the
prediction is relatively easy as the rank of $A_k$ is observed to
be monotonically increasing and become stable at the true rank. So
the prediction rule is:
\begin{eqnarray}\label{eqn:predict}
\mbox{sv}_{k+1}=\left\{
\begin{array}{ll}
\mbox{svp}_{k}+1,&\mbox{if } \mbox{svp}_{k}< \mbox{sv}_{k},\\
\min(\mbox{svp}_{k}+\mbox{round}(0.05d),d),&\mbox{if } \mbox{svp}_{k} = \mbox{sv}_{k},\\
\end{array}
\right.
\end{eqnarray}
where $d=\min(m,n)$, $\mbox{sv}_{k}$ is the predicted dimension
and $\mbox{svp}_{k}$ is the number of singular values in the
$\mbox{sv}_{k}$ singular values that are larger than $\mu_k^{-1}$,
and $\mbox{sv}_{0}=10$. Algorithm \ref{alg:exact_ALM} also uses
the above prediction strategy for the inner loop that solves
$(A_{k+1}^*,E_{k+1}^*)$. For the outer loop, the prediction rule
is simply
$\mbox{sv}_{k+1}=\min(\mbox{svp}_{k}+\mbox{round}(0.1d),d)$. As
for Algorithm \ref{alg:MC-IALM}, the prediction is much more
difficult as the ranks of $A_k$ are often oscillating. It is also
often that for small $k$'s the ranks of $A_k$ are close to $d$ and
then gradually decrease to the true rank, making the partial SVD
inefficient\footnote{Numerical tests show that when we want to
compute more than $0.2d$ principal singular vectors/values, using
PROPACK is often slower than computing the full SVD.}. To remedy
this issue, we initialize both $Y$ and $A$ as zero matrices, and
adopt the following truncation strategy which is similar to that
in \cite{Toh2009}:
\begin{eqnarray}
\mbox{sv}_{k+1}=\left\{
\begin{array}{ll}
\mbox{svn}_{k}+1,&\mbox{if } \mbox{svn}_{k}< \mbox{sv}_{k},\\
\min(\mbox{svn}_{k}+10,d),&\mbox{if } \mbox{svn}_{k} = \mbox{sv}_{k},\\
\end{array}
\right.
\end{eqnarray}
where $\mbox{sv}_{0}=5$ and
\begin{eqnarray}
\mbox{svn}_{k}=\left\{
\begin{array}{ll}
\mbox{svp}_{k},&\mbox{if } \mbox{maxgap}_{k} \leq 2,\\
\min(\mbox{svp}_{k},\mbox{maxid}_{k}),&\mbox{if } \mbox{maxgap}_{k} >2,\\
\end{array}
\right.
\end{eqnarray}
in which $\mbox{maxgap}_{k}$ and $\mbox{maxid}_{k}$ are the
largest ratio between successive singular values (arranging the
computed $\mbox{sv}_k$ singular values in a descending order) and
the corresponding index, respectively. We utilize the gap
information because we have observed that the singular values are
separated into two groups quickly, with large gap between them,
making the rank revealing fast and reliable. With the above
prediction scheme, the rank of $A_k$ becomes monotonically
increasing and be stable at the true rank.

\paragraph{Order of Updating $A$ and $E$.}
Although in theory updating whichever of $A$ and $E$ first does
not affect the convergence rate, numerical tests show that this
does result in slightly different number of iterations to achieve
the same accuracy. Considering the huge complexity of SVD for
large dimensional matrices, such slight difference should also be
considered. Via extensive numerical tests, we suggest updating $E$
first in Algorithms \ref{alg:exact_ALM} and \ref{alg:inexact_ALM}.
What is equally important, updating $E$ first also makes the rank
of $A_k$ much more likely to be monotonically increasing, which is
critical for the partial SVD to be effective, as having been
elaborated in the previous paragraph.

\paragraph{Memory Saving for Algorithm \ref{alg:MC-IALM}.}
In the real implementation of Algorithm \ref{alg:MC-IALM}, sparse
matrices are used to store $D$ and $Y_k$, and as done in
\cite{Toh2009} $A$ is represented as $A=LR^T$, where both $L$ and
$R$ are matrices of size $m\times \mbox{svp}_k$. $E_k$ is not
explicitly stored by observing
\begin{equation}
E_{k+1} = \pi_{\bar{\Omega}}(D-A_{k+1}+\mu_k^{-1}
Y_k)=\pi_{\Omega}(A_{k+1}) - A_{k+1}.\label{eqn:E_k}
\end{equation}
In this way, only $\pi_\Omega(A_k)$ is required to compute $Y_k$
and $D-E_{k}+\mu_k^{-1}Y_k$. So much memory can be saved due to
the small percentage of samples.

\paragraph{Stopping Criteria.} For the RPCA problem, the KKT conditions are:
\begin{equation}
D-A^*-E^*=0, \quad Y^* \in \partial \|A^*\|_*,\quad Y^* \in
\partial( \|\lambda E^*\|_1).
\end{equation}
The last two conditions hold if and only if $\partial \|A^*\|_*
\cap \partial( \|\lambda E^*\|_1) \neq \emptyset$. So we may take
the following conditions as the stopping criteria for Algorithms
\ref{alg:exact_ALM} and \ref{alg:inexact_ALM}:
\begin{equation}
\|D-A_k-E_k\|_F/\|D\|_F < \varepsilon_1\quad\mbox{and}\quad
\mbox{dist}(\partial \|A_k\|_*,\partial( \|\lambda
E_k\|_1))/\|D\|_F < \varepsilon_2,
\end{equation}
where $\mbox{dist}(X,Y)=\min\{\|x-y\|_F|x\in X,y\in Y\}$. For
Algorithm \ref{alg:exact_ALM}, the second condition is always
guaranteed by the inner loop. So we only have to check the first
condition. For Algorithm \ref{alg:inexact_ALM}, unfortunately, it
is expensive to compute $\mbox{dist}(\partial \|A_k\|_*,\partial(
\|\lambda E_k\|_1))$ as the projection onto $\partial \|A_k\|_*$
is costly. So we may estimate $\mbox{dist}(\partial
\|A_k\|_*,\partial( \|\lambda E_k\|_1))$ by $\|\hat{Y}_k -
Y_k\|_F=\mu_{k-1}\|E_{k}-E_{k-1}\|_F$ since $\hat{Y}_k \in
\partial \|A_k\|_*$ and $Y_k\in \partial( \|\lambda E_k\|_1)$.

Similarly, for the MC problem we may take the following conditions
as the stopping criteria for Algorithm \ref{alg:MC-IALM}:
\begin{equation}
\|D-A_k-E_k\|_F/\|D\|_F < \varepsilon_1\quad\mbox{and}\quad
\mbox{dist}(\partial \|A_k\|_*,S)/\|D\|_F < \varepsilon_2,
\end{equation}
where $S=\{Y|Y_{ij}=0 \mbox{ if }(i,j)\notin \Omega\}$. Again, as
it is expensive to compute $\mbox{dist}(\partial \|A_k\|_*,S)$ we
may estimate $\mbox{dist}(\partial \|A_k\|_*,S)$ by $\|\hat{Y}_k -
Y_k\|_F=\mu_{k-1}\|E_{k}-E_{k-1}\|_F$. As
$\mu_{k-1}\|E_{k}-E_{k-1}\|_F$ actually significantly
overestimates $\mbox{dist}(\partial \|A_k\|_*,S)$, in real
computation we may use the following stopping criteria for
Algorithm \ref{alg:MC-IALM}:
\begin{equation}
\|D-A_k-E_k\|_F/\|D\|_F < \varepsilon_1\quad\mbox{and}\quad
\min(\mu_k,\sqrt{\mu_k})\|E_{k}-E_{k-1}\|_F/\|D\|_F <
\varepsilon_2.
\end{equation}
Note that by (\ref{eqn:E_k}) $\|E_{k}-E_{k-1}\|_F$ can be
conveniently computed as $$\sqrt{\|A_k - A_{k-1}\|_F^2 -
\|\pi_{\Omega}(A_k)-\pi_{\Omega}(A_{k-1})\|_F^2}.$$

\paragraph{Updating $\mu_k$.} For the RPCA problem, the updating
rule is:
\begin{eqnarray}
\mu_{k+1}=\left\{\begin{array}{ll} \rho\mu_k, &\mbox{if }
\mu_k\|E_{k+1}-E_{k}\|_F/\|D\|_F < \varepsilon_2,\\
\mu_k, &\mbox{otherwise,}
\end{array}\right.
\end{eqnarray}
where $\rho > 1$. One can easily see that this updating rule is
consistent with Theorem \ref{thm:inexact_ALM}. For the MC problem,
the updating rule is:
\begin{eqnarray}
\mu_{k+1}=\left\{\begin{array}{ll} \rho\mu_k, &\mbox{if }
\min(\mu_k,\sqrt{\mu_k})\|E_{k+1}-E_{k}\|_F/\|D\|_F < \varepsilon_2,\\
\mu_k, &\mbox{otherwise.}
\end{array}\right.
\end{eqnarray}


\paragraph{Choosing Parameters.}
For Algorithm \ref{alg:exact_ALM}, we set
$\mu_0=0.5/\|\sgn(D)\|_2$ and $\rho=6$. The stopping criterion for
the inner loop is $\|A_{k}^{j+1}-A_{k}^{j}\|_F/\|D\|_F < 10^{-6}$
and $\|E_{k}^{j+1}-E_{k}^{j}\|_F/\|D\|_F < 10^{-6}$. The stopping
criterion for the outer iteration is
$\|D-A_{k}^*-E_{k}^*\|_F/\|D\|_F < 10^{-7}$. For Algorithm
\ref{alg:inexact_ALM}, we set $\mu_0=1.25/\|D\|_2$ and $\rho=1.6$.
And the parameters in the stopping criteria are $\varepsilon_1 =
10^{-7}$ and $\varepsilon_2 = 10^{-5}$. For Algorithm
\ref{alg:MC-IALM}, we set $\mu_0=1/\|D\|_2$ and
$\rho=1.2172+1.8588\rho_s$, where $\rho_s=|\Omega|/(mn)$ is the
sampling density and the relation between $\rho$ and $\rho_s$ is
obtained by regression. And the parameters in the stopping
criteria are $\varepsilon_1 = 10^{-7}$ and $\varepsilon_2 =
10^{-6}$.

\section{Simulations}
\label{sec:Experiments} In this section, using numerical
simulations, for the RPCA problem we compare the proposed ALM
algorithms with the APG algorithm proposed in \cite{Lin2009}; for
the MC problem, we compare the inexact ALM algorithm with the SVT
algorithm \cite{Cai2008} and the APG algorithm \cite{Toh2009}. All
the simulations are conducted and timed on the same workstation
with an Intel Xeon E5540 2.53GHz CPU that has 4 cores and 24GB
memory\footnote{But on a Win32 system only 3GB can be used by each
thread.}, running Windows 7 and Matlab (version
7.7).\footnote{Matlab code for all the algorithms compared are
available at
\url{http://perception.csl.illinois.edu/matrix-rank/home.html}}

\paragraph{I. Comparison on the Robust PCA Problem.}
For the RPCA problem, we use randomly generated square matrices
for our simulations. We denote the true solution by the ordered
pair $(A^*,E^*) \in \R^{m\times m} \times \R^{m \times m}$. We
generate the rank-$r$ matrix $A^*$ as a product $LR^T$, where $L$
and $R$ are independent $m \times r$ matrices whose elements are
i.i.d. Gaussian random variables with zero mean and unit
variance.\footnote{It can be shown that $A^*$ is distributed
according to the random orthogonal model of rank $r$, as defined
in \cite{Candes2008}.} We generate $E^*$ as a sparse matrix whose
support is chosen uniformly at random, and whose non-zero entries
are i.i.d. uniformly in the interval $[-500,500]$. The matrix $D
\doteq A^* + E^*$ is the input to the algorithm, and
($\hat{A},\hat{E})$ denotes the output. We choose a fixed
weighting parameter $\lambda = m^{-1/2}$ for a given problem.

We use the latest version of the code for Algorithm \ref{alg:apg},
provide by the authors of \cite{Lin2009}, and also apply the
prediction rule (\ref{eqn:predict}), with $\mbox{sv}_{0}=5$, to it
so that the partial SVD can be utilized\footnote{Such a prediction
scheme was not proposed in \cite{Lin2009}. So the full SVD was
used therein.}. With the partial SVD, APG is faster than the dual
approach in Section \ref{sec:dual}. So we need not involve the
dual approach for comparison.

A brief comparison of the three algorithms is presented in Tables
\ref{tab:rpca_compare1} and \ref{tab:rpca_compare2}. We can see
that both APG and IALM algorithms stop at relatively constant
iteration numbers and IALM is at least five times faster than APG.
Moreover, the accuracies of EALM and IALM are higher than that of
APG. In particular, APG often over estimates $\|E^*\|_0$, the
number of non-zeros in $E^*$, quite a bit. While the estimated
$\|E^*\|_0$ by EALM and IALM are always extremely close to the
ground truth.

\paragraph{II. Comparison on the Matrix Completion Problem.}
For the MC problem, the true low-rank matrix $A^*$ is first
generated as that for the RPCA problem. Then we sample $p$
elements uniformly from $A^*$ to form the known samples in $D$. A
useful quantity for reference is $d_r=r(2m-r)$, which is the
number of degrees of freedom in an $m\times m$ matrix of rank $r$
\cite{Toh2009}.

The SVT and APGL (APG with line search\footnote{For the MC
problem, APGL is faster than APG without line search. However, for
the RPCA problem, APGL is not faster than APG \cite{Lin2009}.})
codes are provided by the authors of \cite{Cai2008} and
\cite{Toh2009}, respectively. A brief comparison of the three
algorithms is presented in Table \ref{tab:mc_compare2}. One can
see that IALM is always faster than SVT. It is also advantageous
over APGL when the sampling density $p/m^2$ is relatively high,
e.g., $p/m^2>10\%$. This phenomenon is actually consistent with
the results on the RPCA problem, where most samples of $D$ are
assumed accurate, although the positions of accurate samples are
not known apriori.

\section{Conclusions}
\label{sec:Conclusions} In this paper, we have proposed two
augmented Lagrange multiplier based algorithms, namely EALM and
IALM, for solving the Robust PCA problem \eqref{eqn:rpca}. Both
algorithms are faster than the previous state-of-the-art APG
algorithm \cite{Lin2009}. In particular, in all simulations IALM
is consistently over five times faster than APG.

We have also applied the method of augmented Lagrange multiplier
to the matrix completion problem. The corresponding IALM algorithm
is considerably faster than the famous SVT algorithm \cite{Cai2008}.
It is also faster than the state-of-the-art APGL algorithm
\cite{Toh2009} when the percentage of available entries is not too
low, say $> 10\%$.

Compared to accelerated proximal gradient based methods, augmented
Lagrange multiplier based algorithms are simpler to analyze and
easier to implement. Moreover, they are also of much higher
accuracy as the iterations are proven to converge to the exact
solution of the problem, even if the penalty parameter does not
approach infinity \cite{Bertsekas-LM}. In contrast, APG methods
normally find a close approximation to the solution by solving a
relaxed problem. Finally, ALM algorithms require less
storage/memory than APG for both the RPCA and MC
problems\footnote{By smart reuse of intermediate matrices (and
accordingly the codes become hard to read), for the RPCA problem
APG still needs one more intermediate (dense) matrix than IALM;
for the MC problem, APG needs two more low rank matrices (for
representing $A_{k-1}$) and one more sparse matrix than IALM. Our
numerical simulation testifies this too: for the MC problem, on
our workstation IALM was able to handle $A^*$ with size
$10^4\times 10^4$ and rank $10^2$, while APG could not.}. For
large-scale applications, such as web data analysis, this could
prove to be a big advantage for ALM type algorithms.

To help the reader to compare and use all the algorithms, we have
posted our Matlab code of all the algorithms at the website:
\begin{quote}
\centerline{\url{http://perception.csl.illinois.edu/matrix-rank/home.html}}
\end{quote}

\begin{table}[p]
\begin{center}
\begin{tabular}{|r|r|rrr|rr|}
\hline
{$m$}    & algorithm & {$\frac{\| \hat A - A^* \|_F}{\|A^*\|_F}$} & rank$(\hat A)$ & $\|\hat E\|_0$  & \#SVD  & time (s)\\
\hline
\multicolumn{6}{c}{} \\
\hline
\multicolumn{7}{|c|}{rank($A^*$) $= 0.05\,m$, $\|E^*\|_0 = 0.05\,m^2$} \\
\hline \hline
500 & APG  & 1.12e-5 & 25 & 12542 & 127 & 11.01 \\
    & EALM  & 3.99e-7 & 25 & 12499 & 28 & 4.08 \\
    & IALM & 5.21e-7 & 25 & 12499 & 20 & 1.72  \\
\hline
800 & APG  & 9.84e-6 & 40 & 32092 & 126 & 37.21 \\
    & EALM  & 1.47e-7 & 40 & 32002 & 29 & 18.59 \\
    & IALM & 3.29e-7 & 40 & 31999 & 21 & 5.87  \\
\hline
1000 & APG  & 8.79e-6 & 50 & 50082 & 126 & 57.62 \\
     & EALM  & 7.85e-8 & 50 & 50000 & 29 & 33.28 \\
     & IALM & 2.67e-7 & 50 & 49999 & 22 & 10.13  \\
\hline
1500 & APG  & 7.16e-6 & 75 & 112659 & 126 & 163.80 \\
     & EALM  & 7.55e-8 & 75 & 112500 & 29 & 104.97 \\
     & IALM & 1.86e-7 & 75 & 112500 & 22 & 30.80  \\
\hline
2000 & APG  & 6.27e-6 & 100 & 200243 & 126 & 353.63 \\
     & EALM  & 4.61e-8 & 100 & 200000 & 30 & 243.64 \\
     & IALM & 9.54e-8 & 100 & 200000 & 22 & 68.69  \\
\hline
3000 & APG  & 5.20e-6 & 150 & 450411 & 126 & 1106.22 \\
     & EALM  & 4.39e-8 & 150 & 449998 & 30 & 764.66 \\
     & IALM & 1.49e-7 & 150 & 449993 & 22 & 212.34  \\
\hline

\hline
\multicolumn{6}{c}{} \\
\hline
\multicolumn{7}{|c|}{rank($A^*$) $= 0.05\,m$, $\|E^*\|_0 = 0.10\,m^2$} \\
\hline \hline
500 & APG  & 1.41e-5 & 25 & 25134 & 129 & 14.35 \\
    & EALM  & 8.72e-7 & 25 & 25009 & 34 & 4.75 \\
    & IALM & 9.31e-7 & 25 & 25000 & 21 & 2.52  \\
\hline
800 & APG  & 1.12e-5 & 40 & 64236 & 129 & 37.94 \\
    & EALM  & 2.86e-7 & 40 & 64002 & 34 & 20.30 \\
    & IALM & 4.87e-7 & 40 & 64000 & 24 & 6.69  \\
\hline
1000 & APG  & 9.97e-6 & 50 & 100343 & 129 & 65.41 \\
     & EALM  & 6.07e-7 & 50 & 100002 & 33 & 30.63 \\
     & IALM & 3.78e-7 & 50 & 99996 & 22 & 10.77  \\
\hline
1500 & APG  & 8.18e-6 & 75 & 225614 & 129 & 163.36 \\
     & EALM  & 1.45e-7 & 75 & 224999 & 33 & 109.54 \\
     & IALM & 2.79e-7 & 75 & 224996 & 23 & 35.71  \\
\hline
2000 & APG  & 7.11e-6 & 100 & 400988 & 129 & 353.30 \\
     & EALM  & 1.23e-7 & 100 & 400001 & 34 & 254.77 \\
     & IALM & 3.31e-7 & 100 & 399993 & 23 & 70.33  \\
\hline
3000 & APG  & 5.79e-6 & 150 & 901974 & 129 & 1110.76 \\
     & EALM  & 1.05e-7 & 150 & 899999 & 34 & 817.69 \\
     & IALM & 2.27e-7 & 150 & 899980 & 23 & 217.39  \\
\hline

\end{tabular}
\end{center}
\caption{{\bf Comparison between APG, EALM and IALM on the Robust
PCA problem.} We present typical running times for randomly
generated matrices. Corresponding to each triplet \{$m$,
rank($A^*$), $\|E^*\|_0$\}, the RPCA problem was solved for the
same data matrix $D$ using three different algorithms. For APG and
IALM, the number of SVDs is equal to the number of
iterations.}\label{tab:rpca_compare1}
\end{table}

\begin{table}[p]
\begin{center}
\begin{tabular}{|r|r|rrr|rr|}
\hline
{$m$}    & algorithm & {$\frac{\| \hat A - A^* \|_F}{\|A^*\|_F}$} & rank$(\hat A)$ & $\|\hat E\|_0$  & \#SVD  & time (s)\\
\hline
\multicolumn{6}{c}{} \\
\hline
\multicolumn{7}{|c|}{rank($A^*$) $= 0.10\,m$, $\|E^*\|_0 = 0.05\,m^2$} \\
\hline \hline
500 & APG  & 9.36e-6 & 50 & 13722 & 129 & 13.99 \\
    & EALM  & 5.53e-7 & 50 & 12670 & 41 & 7.35 \\
    & IALM & 6.05e-7 & 50 & 12500 & 22 & 2.32  \\
\hline
800 & APG  & 7.45e-6 & 80 & 34789 & 129 & 67.54 \\
    & EALM  & 1.13e-7 & 80 & 32100 & 40 & 30.56 \\
    & IALM & 3.08e-7 & 80 & 32000 & 22 & 10.81  \\
\hline
1000 & APG  & 6.64e-6 & 100 & 54128 & 129 & 129.40 \\
     & EALM  & 4.20e-7 & 100 & 50207 & 39 & 50.31 \\
     & IALM & 2.61e-7 & 100 & 50000 & 22 & 20.71 \\
\hline
1500 & APG  & 5.43e-6 & 150 & 121636 & 129 & 381.52 \\
     & EALM  & 1.22e-7 & 150 & 112845 & 41 & 181.28 \\
     & IALM & 1.76e-7 & 150 & 112496 & 24 & 67.84  \\
\hline
2000 & APG  & 4.77e-6 & 200 & 215874 & 129 & 888.93 \\
     & EALM  & 1.15e-7 & 200 & 200512 & 41 & 423.83 \\
     & IALM & 2.49e-7 & 200 & 199998 & 23 & 150.35  \\
\hline
3000 & APG  & 3.98e-6 & 300 & 484664 & 129 & 2923.90 \\
     & EALM  & 7.92e-8 & 300 & 451112 & 42 & 1444.74 \\
     & IALM & 1.30e-7 & 300 & 450000 & 23 & 485.70  \\
\hline

\hline
\multicolumn{6}{c}{} \\
\hline
\multicolumn{7}{|c|}{rank($A^*$) $= 0.10\,m$, $\|E^*\|_0 = 0.10\,m^2$} \\
\hline \hline
500 & APG  & 9.78e-6 & 50 & 27478 & 133 & 13.90 \\
    & EALM  & 1.14e-6 & 50 & 26577 & 52 & 9.46 \\
    & IALM & 7.64e-7 & 50 & 25000 & 25 & 2.62  \\
\hline
800 & APG  & 8.66e-6 & 80 & 70384 & 132 & 68.12 \\
    & EALM  & 3.59e-7 & 80 & 66781 & 51 & 41.33 \\
    & IALM & 4.77e-7 & 80 & 64000 & 25 & 11.88  \\
\hline
1000 & APG  & 7.75e-6 & 100 & 109632 & 132 & 130.37 \\
     & EALM  & 3.40e-7 & 100 & 104298 & 49 & 77.26 \\
     & IALM & 3.73e-7 & 100 & 99999 & 25 & 22.95  \\
\hline
1500 & APG  & 6.31e-6 & 150 & 246187 & 132 & 383.28 \\
     & EALM  & 3.55e-7 & 150 & 231438 & 49 & 239.62 \\
     & IALM & 5.42e-7 & 150 & 224998 & 24 & 66.78  \\
\hline
2000 & APG  & 5.49e-6 & 200 & 437099 & 132 & 884.86 \\
     & EALM  & 2.81e-7 & 200 & 410384 & 51 & 570.72 \\
     & IALM & 4.27e-7 & 200 & 399999 & 24 & 154.27  \\
\hline
3000 & APG  & 4.50e-6 & 300 & 980933 & 132 & 2915.40 \\
     & EALM  & 2.02e-7 & 300 & 915877 & 51 & 1904.95 \\
     & IALM & 3.39e-7 & 300 & 899990 & 24 & 503.05 \\
\hline

\end{tabular}
\end{center}
\caption{{\bf Comparison between APG, EALM and IALM on the Robust
PCA problem.} Continued from Table \ref{tab:rpca_compare2} with
different parameters of \{$m$, rank($A^*$),
$\|E^*\|_0$\}.}\label{tab:rpca_compare2}
\end{table}

\begin{table}[p!]
\begin{center}
\begin{tabular}{|rrrr|r|rrrr|}
\hline
{$m$} & $r$ & $p/d_r$ & $p/m^2$ & algorithm & \#iter & rank$(\hat A)$   & time (s) & {$\frac{\| \hat A - A^* \|_F}{\|A^*\|_F}$} \\
\hline \hline
1000 & 10 & 6 & 0.12 & SVT  & 208 & 10 & 18.23 & 1.64e-6 \\
     &    &   &      & APGL & 69 & 10 & 4.46 & 3.16e-6 \\
     &    &   &      & IALM & 69 & 10 & 3.73 & 1.40e-6 \\
\hline
1000 & 50 & 4 & 0.39 & SVT  & 201 & 50 & 126.18 & 1.61e-6 \\
     &    &   &      & APGL & 76 & 50 & 24.54 & 4.31e-6 \\
     &    &   &      & IALM & 38 & 50 & 12.68 & 1.53e-6 \\
\hline
1000 & 100 & 3 & 0.57 & SVT  & 228 & 100 & 319.93 & 1.71e-6  \\
     &     &   &      & APGL & 81 & 100 & 70.59 & 4.40e-6  \\
     &     &   &      & IALM & 41 & 100 & 42.94 & 1.54e-6 \\
\hline
3000 & 10 & 6 & 0.04 & SVT  & 218 & 10 & 70.14 & 1.77e-6 \\
     &    &   &      & APGL & 88 & 10 & 15.63 & 2.33e-6 \\
     &    &   &      & IALM & 131 & 10 & 27.18 & 1.41e-6 \\
\hline
3000 & 50 & 5 & 0.165 & SVT & 182 & 50 & 370.13 & 1.58e-6 \\
     &    &   &      & APGL & 78 & 50 & 101.04 & 5.74e-6 \\
     &    &   &      & IALM & 57 & 50 & 82.68 & 1.31e-6 \\
\hline
3000 & 100 & 4 & 0.26 & SVT & 204 & 100 & 950.01 & 1.68e-6 \\
     &    &   &      & APGL & 82 & 100 & 248.16 & 5.18e-6 \\
     &     &   &     & IALM & 50 & 100 & 188.22 & 1.52e-6  \\
\hline
5000 & 10 & 6 & 0.024 & SVT & 231 & 10 & 141.88 & 1.79e-6 \\
     &    &   &      & APGL & 81 & 10 & 30.52 & 5.26e-6 \\
     &    &   &      & IALM & 166 & 10 & 68.38 & 1.37e-6 \\
\hline
5000 & 50 & 5 & 0.10 & SVT & 188 & 50 & 637.97 & 1.62e-6 \\
     &    &   &      & APGL & 88 & 50 & 208.08 & 1.93e-6 \\
     &    &   &      & IALM & 79 & 50 & 230.73 & 1.30e-6 \\
\hline
5000 & 100 & 4 & 0.158 & SVT & 215 & 100 & 2287.72 & 1.72e-6 \\
     &     &   &      & APGL & 98 & 100 & 606.82 & 4.42e-6 \\
     &     &   &      & IALM & 64 & 100 & 457.79 & 1.53e-6 \\
\hline
8000 & 10 & 6 & 0.015 & SVT & 230 & 10 & 283.94 & 1.86e-6 \\
     &    &   &      & APGL & 87 & 10 & 66.45 & 5.27e-6 \\
     &    &   &      & IALM & 235 & 10 & 186.73 & 2.08e-6 \\
\hline
8000 & 50 & 5 & 0.06 & SVT & 191 & 50 & 1095.10 & 1.61e-6 \\
     &    &   &      & APGL & 100 & 50 & 509.78 & 6.16e-6 \\
     &    &   &      & IALM & 104 & 50 & 559.22 & 1.36e-6 \\
\hline
10000 & 10 & 6 & 0.012 & SVT & 228 & 10 & 350.20 & 1.80e-6  \\
      &    &   &      & APGL & 89  & 10 & 96.10 & 5.13e-6 \\
      &    &   &      & IALM & 274 & 10 & 311.46 & 1.96e-6 \\
\hline
10000 & 50 & 5 & 0.05 & SVT & 192 & 50 & 1582.95 & 1.62e-6 \\
      &    &   &      & APGL & 105 & 50 & 721.96 & 3.82e-6 \\
      &    &   &      & IALM & 118 & 50 & 912.61 & 1.32e-6 \\
\hline
\end{tabular}
\end{center}
\caption{{\bf Comparison between SVT, APG and IALM on the matrix
completion problem.} We present typical running times for randomly
generated matrices. Corresponding to each triplet \{$m$,
rank($A^*$), $p/d_r$\}, the MC problem was solved for the same
data matrix $D$ using the three different algorithms.
}\label{tab:mc_compare2}
\end{table}

\begin{acknowledgements}
We thank Leqin WU for contributing to the second part of the proof
of Theorem \ref{thm:exact_ALM} and helping draft the early version
of the paper. We thank the authors of \cite{Toh2009} for kindly
sharing with us their code of APG and APGL for matrix completion.
We would also like to thank Arvind Ganesh of UIUC and Dr. John
Wright of MSRA for providing the code of APG for matrix recovery.
\end{acknowledgements}

\appendix

\section{Proofs and Technical Details for Section
\ref{sec:ALM}}

In this appendix, we provide the mathematical details in Section
\ref{sec:ALM}. To prove Theorems \ref{thm:exact_ALM} and
\ref{thm:inexact_ALM}, we have to prepare some results in Sections
\ref{sec:duality} and \ref{sec:boundedness}.

\subsection{Relationship between Primal and Dual Norms}
\label{sec:duality} Our convergence theorems require the
boundedness of some sequences, which results from the following
theorem.
\begin{theorem} \label{thm:dual_norm} Let $\mathcal{H}$ be
a real Hilbert space endowed with an inner product $\langle
\cdot,\cdot \rangle$ and a corresponding norm $\|\cdot\|$, and $y
\in
\partial \|x\|$, where $\partial f(x)$ is the subgradient of $f(x)$. Then $\|y\|^*=1$ if $x\neq 0$, and $\|y\|^*\leq 1$ if $x= 0$, where $\|\cdot\|^*$ is the dual norm of
$\|\cdot\|$.
\end{theorem}
\begin{proof} As $y \in \partial \|x\|$, we have
\begin{equation}\label{eqn:subgradient1}
\|w\|-\|x\| \geq \< y,w-x \>, \quad \forall~ w\in \mathcal{H}.
\end{equation}
If $x\neq 0$, choosing $w=0,2x$, we can deduce that
\begin{equation}\label{eqn:equal}
\|x\| = \< y,x \> \leq \|x\|\|y\|^*.
\end{equation}
So $\|y\|^* \geq 1$. On the other hand, we have
\begin{equation}
 \|w-x\| \geq \|w\|-\|x\| \geq \< y,w-x \>,\quad
 \forall~w\in \mathcal{H}.
\end{equation}
So
$$
\< y,\frac{w-x}{\|w-x\|} \> \leq 1,\quad\forall~w \neq x.
$$
Therefore $\|y\|^* \leq 1$. Then we conclude that $\|y\|^*=1$.

If $x=0$, then (\ref{eqn:subgradient1}) is equivalent to
\begin{equation}\label{eqn:subgradient2}
\< y,w \> \leq 1, \quad \forall~ \|w\|=1.
\end{equation}
By the definition of dual norm, this means that $\|y\|^*\leq 1$.
\qquad \end{proof}

\subsection{Boundedness of Some Sequences} \label{sec:boundedness}
With Theorem \ref{thm:dual_norm}, we can prove the following
lemmas.
\begin{lemma}\label{lemma:bound_Y}
The sequences $\{Y_k^*\}$, $\{Y_k\}$ and $\{\hat{Y}_{k}\}$ are all
bounded, where $\hat{Y}_{k}=Y_{k-1}+\mu_{k-1}(D-A_{k}-E_{k-1})$.
\end{lemma}
\begin{proof}
By the optimality of $A_{k+1}^*$ and $E_{k+1}^*$ we have that:
\begin{equation}
0\in \partial_A L(A_{k+1}^*,E_{k+1}^*,Y_k^*,\mu_k),\quad 0\in
\partial_E L(A_{k+1}^*,E_{k+1}^*,Y_k^*,\mu_k),
\end{equation}
i.e.,
\begin{equation}
\begin{array}{l}
0\in \partial \|A_{k+1}^*\|_* - Y_k^* -
\mu_k(D-A_{k+1}^*-E_{k+1}^*),\\ 0\in
\partial \left(\|\lambda E_{k+1}^*\|_1\right) - Y_k^* -
\mu_k(D-A_{k+1}^*-E_{k+1}^*).
\end{array}
\end{equation}
So we have that
\begin{equation}
Y_{k+1}^*\in \partial \|A_{k+1}^*\|_*, \quad Y_{k+1}^* \in
\partial \left(\|\lambda E_{k+1}^*\|_1\right).
\end{equation}
Then by Theorem \ref{thm:dual_norm} the sequences $\{Y_k^*\}$ is
bounded\footnote{A stronger result is that
$\|Y_k^*\|_2=\lambda^{-1}\|Y_k^*\|_\infty =1$ if $A_{k}^*\neq 0$
and $E_{k}^* \neq 0$.} by observing the fact that the dual norms
of $\|\cdot\|_*$ and $\|\cdot\|_1$ are $\|\cdot\|_2$ and
$\|\cdot\|_\infty$ \cite{Cai2008,Lin2009}, respectively. The
boundedness of $\{Y_k\}$ and $\{\hat{Y}_{k}\}$ can be proved
similarly. \qquad
\end{proof}

\subsection{Proof of Theorem
\ref{thm:exact_ALM}}\label{sec:exact_ALM} \
\begin{proof} By
\begin{equation}
\begin{array}{rcl}
L(A_{k+1}^*,E_{k+1}^*,Y_k^*,\mu_k)&=&\min\limits_{A,E}
L(A,E,Y_k^*,\mu_k)\\
&\leq &\min\limits_{A+E=D}
L(A,E,Y_k^*,\mu_k)\\
&=&\min\limits_{A+E=D}\left(\|A\|_*+\lambda \|E\|_1\right)=f^*,
\end{array}
\end{equation}
we have
\begin{equation}
\begin{array}{rcl}
&&\|A_{k+1}^*\|_*+\lambda \|E_{k+1}^*\|_1\\
&=&L(A_{k+1}^*,E_{k+1}^*,Y_k^*,\mu_k) -
\dfrac{1}{2\mu_k}\left(\|Y_{k+1}^*\|_F^2 - \|Y_k^*\|_F^2\right)\\
&\leq & f^*-\dfrac{1}{2\mu_k}\left(\|Y_{k+1}^*\|_F^2 -
\|Y_k^*\|_F^2\right).
\end{array}
\end{equation}
By the boundedness of $\{Y_{k}^*\}$, we see that
\begin{equation}\label{eqn:upper_bound}
\|A_{k+1}^*\|_*+\lambda \|E_{k+1}^*\|_1 \leq f^*+O(\mu_k^{-1}).
\end{equation}
By letting $k \rightarrow +\infty$, we have that
\begin{equation}
\|A^*\|_*+\lambda \|E^*\|_1 \leq f^*.
\end{equation}
As $D - A_{k+1}^* - E_{k+1}^* = \mu_k^{-1}(Y_{k+1}^* - Y_{k}^*)$,
by the boundedness of $Y_{k}^*$ and letting $k \rightarrow
+\infty$ we see that
\begin{equation}
A^*+E^*=D.
\end{equation}
Therefore, $(A^*,E^*)$ is an optimal solution to the RPCA problem.

On the other hand, by the triangular inequality of norms,
\begin{equation}
\begin{array}{rcl}
\|A_{k+1}^*\|_*+\lambda \|E_{k+1}^*\|_1 &\geq& \|D-E_{k+1}^*\|_* +
\lambda \|E_{k+1}^*\|_1 - \|D-A_{k+1}^* - E_{k+1}^*\|_*\\
&\geq & f^* - \|D-A_{k+1}^* - E_{k+1}^*\|_*\\
&=&f^* - \mu_k^{-1}\|Y_{k+1}^* - Y_{k}^*\|_*.
\end{array}
\end{equation}
So
\begin{equation} \|A_{k+1}^*\|_*+\lambda \|E_{k+1}^*\|_1 \geq f^*
- O(\mu_k^{-1}).
\end{equation}
This together with (\ref{eqn:upper_bound}) proves the convergence
rate. \qquad \end{proof}

\subsection{Key Lemmas}\label{sec:recursive_relationship} First, we have the following
identity.
\begin{lemma}\label{lemma:recursive_relationship}
\begin{eqnarray}
\begin{array}{rl}
&\|E_{k+1}-E^*\|_F^2+\mu_{k}^{-2}\|Y_{k+1}-Y^*\|_F^2 \\
= &
\|E_{k}-E^*\|_F^2+\mu_{k}^{-2}\|Y_{k}-Y^*\|_F^2-\|E_{k+1}-E_k\|_F^2-
\mu_k^{-2} \|Y_{k+1}-Y_k\|_F^2\\
& -2\mu_k^{-1}(\langle Y_{k+1}-Y_k,E_{k+1}-E_k\rangle+\langle
A_{k+1}-A^*,\hat{Y}_{k+1}-Y^*\rangle+\langle
E_{k+1}-E^*,Y_{k+1}-Y^*\rangle),
\end{array}\label{eqn:identity}
\end{eqnarray}
where $(A^*,E^*)$ and $Y^*$ are the optimal solutions to the RPCA
problem (\ref{eqn:rpca}) and the dual problem (\ref{eqn:Dual}),
respectively.
\end{lemma}
\begin{proof} The identity can be routinely checked.
Using $A^*+E^*=D$ and $D-A_{k+1}-E_{k+1}=\mu_k^{-1}(Y_{k+1}-Y_k)$,
we have
\begin{eqnarray}
\begin{array}{rl}
&\mu_k^{-1}\langle Y_{k+1}-Y_k,Y_{k+1}-Y^*\rangle \\
=&-\langle A_{k+1}-A^*,Y_{k+1}-Y^*\rangle - \langle
E_{k+1}-E^*,Y_{k+1}-Y^*\rangle\\
=& \langle A_{k+1}-A^*,\hat{Y}_{k+1}-Y_{k+1}\rangle - \langle
A_{k+1}-A^*,\hat{Y}_{k+1}-Y^*\rangle- \langle
E_{k+1}-E^*,Y_{k+1}-Y^*\rangle\\
=&\mu_k  \langle A_{k+1}-A^*,E_{k+1}-E_{k}\rangle - \langle
A_{k+1}-A^*,\hat{Y}_{k+1}-Y^*\rangle- \langle
E_{k+1}-E^*,Y_{k+1}-Y^*\rangle.
\end{array}
\label{eqn:monotone1}
\end{eqnarray}
Then
\begin{eqnarray*}
&&\|E_{k+1}-E^*\|_F^2+\mu_k^{-2}\|Y_{k+1}-Y^*\|_F^2 \\
&=& (\|E_k-E^*\|_F^2 - \|E_{k+1}-E_k\|_F^2 +2\langle
E_{k+1}-E^*,E_{k+1}-E_k\rangle)\\
&&+\mu_k^{-2} (\|Y_{k}-Y^*\|_F^2-\|Y_{k+1}-Y_k\|_F^2+2\langle
Y_{k+1}-Y^*,Y_{k+1}-Y_k \rangle)\\
&= & \|E_k-E^*\|_F^2 +\mu_k^{-2} \|Y_{k}-Y^*\|_F^2 -
\|E_{k+1}-E_k\|_F^2-\mu_k^{-2}\|Y_{k+1}-Y_k\|_F^2\\
&&+2\langle E_{k+1}-E^*,E_{k+1}-E_k\rangle+2\mu_k^{-2} \langle
Y_{k+1}-Y^*,Y_{k+1}-Y_k \rangle\\
&=& \|E_k-E^*\|_F^2 +\mu_k^{-2} \|Y_{k}-Y^*\|_F^2 -
\|E_{k+1}-E_k\|_F^2-\mu_k^{-2}\|Y_{k+1}-Y_k\|_F^2\\
&&+2\langle E_{k+1}-E^*,E_{k+1}-E_k\rangle+2\langle A_{k+1}-A^*,E_{k+1}-E_k \rangle\\
&&-2\mu_k^{-1}(\langle
A_{k+1}-A^*,\hat{Y}_{k+1}-Y^*\rangle+\langle
E_{k+1}-E^*,Y_{k+1}-Y^*\rangle)\\
&=& \|E_k-E^*\|_F^2 +\mu_k^{-2} \|Y_{k}-Y^*\|_F^2 -
\|E_{k+1}-E_k\|_F^2-\mu_k^{-2}\|Y_{k+1}-Y_k\|_F^2\\
&&+2\langle A_{k+1}+E_{k+1}-D,E_{k+1}-E_k\rangle\\
&&-2\mu_k^{-1}(\langle
A_{k+1}-A^*,\hat{Y}_{k+1}-Y^*\rangle+\langle
E_{k+1}-E^*,Y_{k+1}-Y^*\rangle)\\
&=&\|E_k-E^*\|_F^2 +\mu_k^{-2} \|Y_{k}-Y^*\|_F^2 -
\|E_{k+1}-E_k\|_F^2-\mu_k^{-2}\|Y_{k+1}-Y_k\|_F^2\\
&&-2\mu_k^{-1}\langle Y_{k+1}-Y_k,E_{k+1}-E_k
\rangle-2\mu_k^{-1}(\langle
A_{k+1}-A^*,\hat{Y}_{k+1}-Y^*\rangle+\langle
E_{k+1}-E^*,Y_{k+1}-Y^*\rangle).\\
\end{eqnarray*}
\qquad
\end{proof}
We then quote a classic result.
\begin{lemma}\label{lemma:monotone}
The subgradient of a convex function is a monotone operator.
Namely, if $f$ is a convex function then
$$\langle x_1-x_2, g_1-g_2\rangle \geq 0, \quad \forall g_i\in \partial f(x_i),i=1,2.$$
\end{lemma}
\begin{proof} By the definition of subgradient, we have
$$f(x_1) - f(x_2) \geq \langle x_1-x_2, g_2\rangle,\quad f(x_2) - f(x_1) \geq \langle x_2-x_1, g_1\rangle.$$
Adding the above two inequalities proves the lemma. \qquad
\end{proof}
Then we have the following result.
\begin{lemma}\label{lemma:finite_series_sum}
If $\mu_k$ is nondecreasing then each entry of the following
series is nonnegative and its sum is finite:
\begin{eqnarray}
\begin{array}{rl}
&\sum\limits_{k=1}^{+\infty}\mu_k^{-1}(\langle
Y_{k+1}-Y_k,E_{k+1}-E_k\rangle+\langle
A_{k+1}-A^*,\hat{Y}_{k+1}-Y^*\rangle+\langle
E_{k+1}-E^*,Y_{k+1}-Y^*\rangle)\\
<&+\infty.
\end{array}\label{eqn:finite_series_sum}
\end{eqnarray}
\end{lemma}
\begin{proof} $(A^*,E^*,Y^*)$ is a saddle point of the Lagrangian function
\begin{equation}
L(A,E,Y)=\|A\|_*+\lambda \|E\|_1 + \langle Y, D-A-E\rangle.
\end{equation}
of the RPCA problem (\ref{eqn:rpca}). So we have
\begin{equation}
Y^*\in\partial \|A^*\|_*,\quad Y^* \in\partial (\|\lambda
E^*\|_1).
\end{equation}
Then by Lemma \ref{lemma:monotone}, $\hat{Y}_{k+1}\in \partial
\|A_{k+1}\|_*$, and $Y_{k+1} \in
\partial \left(\|\lambda E_{k+1}\|_1\right)$ (cf. Section \ref{sec:boundedness}),
we have
\begin{equation}
\begin{array}{rcl}
\langle A_{k+1}-A^*,\hat{Y}_{k+1}-Y^*\rangle &\geq& 0,\\
\langle E_{k+1}-E^*,Y_{k+1}-Y^*\rangle &\geq& 0,\\
\langle E_{k+1}-E_k,Y_{k+1}-Y_k\rangle &\geq& 0.
\end{array}
\end{equation}
The above together with $\mu_{k+1}\geq \mu_k$ and
(\ref{eqn:identity}), we have that
$\{\|E_{k}-E^*\|_F^2+\mu_{k}^{-2}\|Y_{k}-Y^*\|_F^2\}$ is
non-increasing and
\begin{eqnarray}
\begin{array}{rl}
&2\mu_k^{-1}(\langle Y_{k+1}-Y_k,E_{k+1}-E_k\rangle+\langle
A_{k+1}-A^*,\hat{Y}_{k+1}-Y^*\rangle+\langle
E_{k+1}-E^*,Y_{k+1}-Y^*\rangle)\\
\leq & (\|E_{k}-E^*\|_F^2+\mu_{k}^{-2}\|Y_{k}-Y^*\|_F^2) -
(\|E_{k+1}-E^*\|_F^2+\mu_{k+1}^{-2}\|Y_{k+1}-Y^*\|_F^2).
\end{array}
\end{eqnarray}
So (\ref{eqn:finite_series_sum}) is proven. \qquad
\end{proof}

\subsection{Proof of Theorem \ref{thm:inexact_ALM}} \label{sec:inexact_ALM}
\
\begin{proof} When $\{\mu_k\}$ is upper bounded, the convergence
of Algorithm \ref{alg:inexact_ALM} is already proved by He et al.
In the following, we assume that $\mu_k\rightarrow +\infty$.

Similar to the proof of Lemma \ref{lemma:finite_series_sum}, we
have
$$\sum\limits_{k=1}^{+\infty}\mu_k^{-2}\|Y_{k+1}-Y_k\|_F^2<+\infty.$$ So
$$\|D-A_k-E_k\|_F=\mu_k^{-1}\|Y_{k}-Y_{k-1}\|_F\rightarrow 0.$$
Then any accumulation point of $(A_k,E_k)$ is a feasible solution.

On the other hand, denote the optimal objective value of the RPCA
problem by $f^*$. As $ \hat{Y}_{k}\in\partial\|A_{k}\|_*$ and
$Y_{k}\in\partial \left(\lambda \|E_{k}\|_1\right)$, we have
\begin{equation}\label{eqn:upper_bound2}
\begin{array}{rl}
&\|A_{k}\|_*+\lambda \|E_{k}\|_1 \\
\leq& \|A^*\|_*+\lambda \|E^*\|_1 - \langle \hat{Y}_{k}, A^*-A_{k}
\rangle - \langle Y_{k},
E^*-E_{k}  \rangle\\
=& f^* +\langle Y^*-\hat{Y}_{k}, A^*-A_{k} \rangle + \langle
Y^*-Y_{k}, E^*-E_{k}  \rangle - \langle
Y^*,A^*-A_{k}+E^*-E_{k}\rangle\\
=&f^* +\langle Y^*-\hat{Y}_{k}, A^*-A_{k} \rangle + \langle
Y^*-Y_{k}, E^*-E_{k}  \rangle - \langle Y^*,D-A_{k}-E_{k}\rangle.
\end{array}
\end{equation}
From Lemma \ref{lemma:finite_series_sum},
\begin{eqnarray}
\sum\limits_{k=1}^{+\infty}\mu_k^{-1}(\langle
A_{k}-A^*,\hat{Y}_{k}-Y^*\rangle+\langle
E_{k}-E^*,Y_{k}-Y^*\rangle)<+\infty.
\end{eqnarray}
As $\sum\limits_{k=1}^{+\infty}\mu_k^{-1}=+\infty$, there must
exist a subsequence $(A_{k_j},E_{k_j})$ such that
$$\langle
A_{k_j}-A^*,\hat{Y}_{k_j}-Y^*\rangle+\langle
E_{k_j}-E^*,Y_{k_j}-Y^*\rangle \rightarrow 0.$$ Then we see that
$$\lim\limits_{j\rightarrow +\infty}\|A_{k_j}\|_*+\lambda
\|E_{k_j}\|_1\leq f^*.$$ So $(A_{k_j},E_{k_j})$ approaches to an
optimal solution $(A^*,E^*)$ to the RPCA problem. As
$\mu_k\rightarrow +\infty$ and $\{Y_k\}$ are bounded, we have that
$\{\|E_{k_j}-E^*\|_F^2+\mu_{k_j}^{-2}\|Y_{k_j}-Y^*\|_F^2\}
\rightarrow 0$.

On the other hand, in the proof of Lemma
\ref{lemma:finite_series_sum} we have shown that
$\{\|E_{k}-E^*\|_F^2+\mu_{k}^{-2}\|Y_{k}-Y^*\|_F^2\}$ is
non-increasing. So
$\|E_{k}-E^*\|_F^2+\mu_{k}^{-2}\|Y_{k}-Y^*\|_F^2 \rightarrow 0$
and we have that $\lim\limits_{k\rightarrow +\infty} E_k=E^*$. As
$\lim\limits_{k\rightarrow+\infty}D-A_{k}-E_{k}=0$ and
$D=A^*+E^*$, we see that $\lim\limits_{k\rightarrow +\infty}
A_k=A^*$. \qquad
\end{proof}

\subsection{Proof of Theorem \ref{thm:necessary_condition}} \label{sec:necessary_condition}
\
\begin{proof} In Lemma \ref{lemma:bound_Y}, we have proved that both $\{Y_k\}$ and
$\{\hat{Y}_k\}$ are bounded sequences. So there exists a constant
$C$ such that
$$\|Y_k\|_F\leq C,\quad\mbox{and}\quad \|\hat{Y}_k\|_F\leq C.$$
Then $$\|E_{k+1}-E_k\|_F =
\mu_k^{-1}\|\hat{Y}_{k+1}-Y_{k+1}\|_F\leq 2C\mu_k^{-1}.
$$
As $\sum\limits_{k=0}^\infty \mu_k^{-1} < +\infty$, we see that
$\{E_k\}$ is a Cauchy sequence, hence it has a limit $E_{\infty}$.
Then
\begin{eqnarray}
\begin{array}{rl}
\|E_{\infty} - E^*\|_F & =
\left\|E_0+\sum\limits_{k=0}^\infty(E_{k+1}-E_k)-E^*\right\|_F\\
&\geq \|E_0-E^*\|_F -\sum\limits_{k=0}^\infty\|E_{k+1}-E_k\|_F\\
&\geq \|E_0-E^*\|_F -2C\sum\limits_{k=0}^\infty \mu_k^{-1}.
\end{array}
\end{eqnarray}
So if Algorithm \ref{alg:inexact_ALM} is badly initialized such
that $$\|E_0-E^*\|_F
> 2C\sum\limits_{k=0}^\infty \mu_k^{-1}$$
then $E_k$ will not converge to $E^*$.
\end{proof}

\bibliographystyle{spmpsci}      
\bibliography{rpca_algorithms}   

\begin{thebibliography}{10}
\providecommand{\url}[1]{{#1}}
\providecommand{\urlprefix}{URL }
\expandafter\ifx\csname urlstyle\endcsname\relax
  \providecommand{\doi}[1]{DOI~\discretionary{}{}{}#1}\else
  \providecommand{\doi}{DOI~\discretionary{}{}{}\begingroup
  \urlstyle{rm}\Url}\fi

\bibitem{Abernethy2006}
Abernethy, J., Bach, F., Evgeniou, T., Vert, J.P.: Low-rank matrix
  factorization with attributes.
\newblock Ecole des Mines de Paris, Technical report, N24/06/MM (2006)

\bibitem{Amit2007}
Amit, Y., Fink, M., Srebro, N., Ullman, S.: Uncovering shared structures in
  multiclass classification.
\newblock In: Proceedings of the Twenty-fourth International Conference on
  Machine Learning (2007)

\bibitem{Argyriou2007}
Argyriou, A., Evgeniou, T., Pontil, M.: Multi-task feature learning.
\newblock In: Proceedings of Advances in Neural Information Processing Systems
  (2007)

\bibitem{Beck2009}
Beck, A., Teboulle, M.: A fast iterative shrinkage-thresholding algorithm for
  linear inverse problems.
\newblock {SIAM} Journal on Imaging Sciences \textbf{2}(1), 183--202 (2009)

\bibitem{Bertsekas-LM}
Bertsekas, D.: Constrained Optimization and Lagrange Multiplier Method.
\newblock Academic Press (1982)

\bibitem{Bertsekas-NP}
Bertsekas, D.: Nonlinear Programming.
\newblock Athena Scientific (1999)

\bibitem{Cai2008}
Cai, J., Cand{\`{e}s}, E., Shen, Z.: A singular value thresholding algorithm
  for matrix completion.
\newblock preprint, code available at http://svt.caltech.edu/code.html  (2008)

\bibitem{Cai2009-MC}
Cai, J.F., Osher, S., Shen, Z.: Linearized {Bregman} iterations for compressed
  sensing.
\newblock Math. Comp. \textbf{78}, 1515--1536 (2009)

\bibitem{Candes2008}
Cand{\`{e}s}, E., Recht, B.: Exact matrix completion via convex optimization.
\newblock {\rm preprint}  (2008)

\bibitem{Chandrasekharan2009}
Chandrasekharan, V., Sanghavi, S., Parillo, P., Wilsky, A.: Rank-sparsity
  incoherence for matrix decomposition.
\newblock {\rm preprint}  (2009)

\bibitem{Fukushima1981-IJSS}
Fukushima, M., Mine, H.: A generalized proximal gradient algorithm for certain
  nonconvex minimization problems.
\newblock International Journal of Systems Science \textbf{12}, 989--1000
  (1981)

\bibitem{Boyd-cvx}
Grant, M., Boyd, S.: {CVX}: {M}atlab software for disciplined convex
  programming (web page and software).
\newblock http://stanford.edu/$\sim$boyd/cvx (2009)

\bibitem{Jolliffe1986}
Jolliffe, I.T.: Principal Component Analysis.
\newblock Springer-Verlag (1986)

\bibitem{Propack}
Larsen, R.M.: Lanczos bidiagonalization with partial reorthogonalization.
\newblock Department of Computer Science, Aarhus University, Technical report,
  DAIMI PB-357, code available at http://soi.stanford.edu/$\sim$rmunk/PROPACK/
  (1998)

\bibitem{Lin2009}
Lin, Z., Ganesh, A., Wright, J., Wu, L., Chen, M., Ma, Y.: Fast convex
  optimization algorithms for exact recovery of a corrupted low-rank matrix.
\newblock {SIAM} J. Optimization  (submitted)

\bibitem{Mesbahi1997}
Mesbahi, M., Papavassilopoulos, G.P.: On the rank minimization problem over a
  positive semidefinite linear matrix inequality.
\newblock {IEEE} Transactions on Automatic Control \textbf{42}(2), 239--243
  (1997)

\bibitem{Nesterov1983-SMD}
Nesterov, Y.: A method of solving a convex programming problem with convergence
  rate \textit{O}$(1/k^2)$.
\newblock Soviet Mathematics Doklady \textbf{27}(2), 372--376 (1983)

\bibitem{Recht2008-SR}
Recht, B., Fazel, M., Parillo, P.: Guaranteed minimum rank solution of matrix
  equations via nuclear norm minimization.
\newblock {\rm submitted to} SIAM Review  (2008)

\bibitem{Toh2009}
Toh, K.C., Yun, S.: An accelerated proximal gradient algorithm for nuclear norm
  regularized least squares problems.
\newblock {\rm preprint}  (2009)

\bibitem{Tomasi1992}
Tomasi, C., Kanade, T.: Shape and motion from image streams under orthography:
  a factorization method.
\newblock International Journal of Computer Vision \textbf{9}(2), 137--154
  (1992)

\bibitem{Tseng2008-SJO}
Tseng, P.: On accelerated proximal gradient methods for convex-concave
  optimization.
\newblock {\rm submitted to} {SIAM} Journal on Optimization  (2008)

\bibitem{Wright2009}
Wright, J., Ganesh, A., Rao, S., Ma, Y.: Robust principal component analysis:
  Exact recovery of corrupted low-rank matrices via convex optimization.
\newblock {\rm submitted to} Journal of the ACM  (2009)

\bibitem{Wright-NIPS2009}
Wright, J., Ganesh, A., Rao, S., Peng, Y., Ma, Y.: Robust principal component
  analysis: Exact recovery of corrupted low-rank matrices via convex
  optimization.
\newblock In: Proceedings of Advances in Neural Information Processing Systems
  (2009)

\bibitem{Yin2008}
Yin, W., Hale, E., Zhang, Y.: Fixed-point continuation for
  $\ell_1$-minimization: methodology and convergence.
\newblock {\rm preprint}  (2008)

\bibitem{Yin2008-SJIS}
Yin, W., Osher, S., Goldfarb, D., Darbon, J.: Bregman iterative algorithms for
  $\ell_1$-minimization with applications to compressed sensing.
\newblock {SIAM} Journal on Imaging Sciences \textbf{1}(1), 143--168 (2008)

\end{thebibliography}

%
%

\end{document}